\documentclass[a4paper,11pt]{article}

\usepackage{amsmath}
\usepackage{amssymb}
\usepackage{amsthm}
\usepackage{graphicx}
\usepackage{amscd}
\usepackage{epic, eepic}
\usepackage{url}
\usepackage{color}
\usepackage[utf8]{inputenc} 
\usepackage{comment}
\usepackage{enumerate}   
\usepackage{todonotes}
\usepackage{graphicx}
\usepackage{epstopdf}
\usepackage{enumitem}
\usepackage{tikz}
\usepackage[top=22mm, bottom=23mm, left=23mm, right=23mm]{geometry}
\usepackage[bookmarks=false,hidelinks]{hyperref}

\makeatletter
\renewenvironment{proof}[1][\proofname] {\par\pushQED{\qed}\normalfont\topsep6\p@\@plus6\p@\relax\trivlist\item[\hskip\labelsep\bfseries#1\@addpunct{.}]\ignorespaces}{\popQED\endtrivlist\@endpefalse}
\makeatother

\newtheorem{theorem}{\bf Theorem}[section]
\newtheorem{lemma}[theorem]{\bf Lemma}

\newtheorem{conjecture}[theorem]{\bf Conjecture}

\theoremstyle{definition}

\newtheorem{definition}[theorem]{\bf Definition}

\def\eps{\varepsilon}
\def\vu{\textbf{u}}
\def\vx{\textbf{x}}
\def\vy{\textbf{y}}

\def\cF{\mathcal{F}}
\def\cG{\mathcal{G}}
\def\cH{\mathcal{H}}

\def\ex{\mathrm{ex}}

\title{\vspace{-0.9cm}The Tur\'an number of the grid}
\date{}

\author{
Domagoj Brada\v{c}\thanks{Department of Mathematics, ETH, Z\"urich, Switzerland. Research supported in part by SNSF grant 200021\_196965. Email: \textbf{\{domagoj.bradac, benjamin.sudakov, istvan.tomon\}@math.ethz.ch}.}
\and 
Oliver Janzer\thanks{Department of Mathematics, ETH, Z\"urich, Switzerland. Research supported by an ETH Z\"urich Postdoctoral Fellowship 20-1 FEL-35. Email: \textbf{oliver.janzer@math.ethz.ch}.}
\and
Benny Sudakov\footnotemark[1]
\and
Istv\'{a}n Tomon\footnotemark[1]
}

\begin{document}

\maketitle

\begin{abstract}
    For a positive integer $t$, let $F_t$ denote the graph of the $t\times t$ grid. Motivated by a 50-year-old conjecture of Erd\H{o}s about Tur\'{a}n numbers of $r$-degenerate graphs, we prove that there exists a constant $C=C(t)$ such that $\ex(n,F_t)\leq Cn^{3/2}$. This bound is tight up to the value of~$C$. One of the interesting ingredients of our proof is a novel way of using the tensor power trick.
\end{abstract}

\section{Introduction}

For a graph $H$ and a positive integer $n$, the Tur\'an number (or extremal number), denoted $\ex(n,H)$, is the maximum number of edges in an $H$-free graph on $n$ vertices. Estimating this function for various choices of $H$ is one of the most important problems in extremal graph theory. The celebrated Erd\H os--Stone--Simonovits theorem \cite{ES46,ESi66} states that
$\ex(n,H)=\left(1-\frac{1}{\chi(H)-1}+o(1)\right)\binom{n}{2},$
which determines the asymptotics of $\ex(n,H)$ whenever $\chi(H)\geq 3$. For bipartite graphs the formula only gives $\ex(n,H)=o(n^2)$, but it is known \cite{KST54} that in this case there is some $\eps=\eps(H)>0$ such that $\ex(n,H)=O(n^{2-\eps})$. It is conjectured that for every graph there is a (rational) number $\alpha$ such that $\ex(n,H)=\Theta(n^{\alpha})$. However, there are relatively few bipartite graphs $H$ for which this is known. The existence of a suitable $\alpha$ has not been established even for some very simple graphs such as the cycle $C_8$, the complete bipartite graph $K_{4,4}$ and the cube $Q_3$. For a survey on the topic, we refer the reader to \cite{FS13}.

In general, it is not even clear which graph parameters determine the asymptotics of the Tur\'{a}n number of a bipartite graph. Nevertheless, in 1966, Erd\H{o}s conjectured that $\ex(n,H) = O(n^{2-1/r})$ holds for every $r$-degenerate bipartite graph $H$. Such estimate if true would be best possible. F\"uredi \cite{Fur91} proved a special case of the conjecture when $H$ has maximum degree $r$ on one side of the bipartition. Alon, Krivelevich and Sudakov \cite{AKS03} gave a new proof of this result, using a powerful probabilistic technique, called dependent random choice (see e.g., \cite{FS} for a description of this method and a brief history). They also proved a bound of $O(n^{2-\frac{1}{4r}})$ for general $r$-degenerate $H$. However, Erd\H{o}s' conjecture in full generality remains widely open even in the case $r=2$. 

In this paper we study Tur\'an numbers of grids. For an integer $t\geq 2$, the grid graph $F_t$ has vertex set $[t]\times [t]$ with two vertices joined by an edge if they differ in exactly one coordinate and in that coordinate they differ by exactly one. It is not difficult to see that $F_t$ is a $2$-degenerate graph, and therefore by Erd\H{o}s' conjecture, one can expect its Tur\'an number to be $O(n^{3/2})$. Recently, extremal problems involving grids have attracted considerable attention. Let us briefly mention a few of these papers. Clemens, Miralaei, Reding, Schacht and Taraz \cite{CMRST} gave an upper bound for the size Ramsey number of the $n\times n$ grid. Their bound was improved very recently by Conlon, Nenadov and Truji\'c~\cite{CNT22}. Kim, Lee and Lee \cite{KLL16} proved Sidorenko's conjecture for grids (in arbitrary dimension). F\"uredi and Ruszink\'o \cite{FR13} studied the maximum number of hyperedges that an $r$-uniform linear hypergraph can have without containing a certain $r\times r$ hypergraph grid; see also \cite{GS22}.

Our main result determines $\ex(n,F_t)$ up to a multiplicative constant, adding grids to the rather short list of families of bipartite graphs whose Tur\'an number is known.

\begin{theorem} \label{thm:turan for grid}
    For any integer $t\geq 2$, there exist positive real numbers $c=c(t)$ and $C=C(t)$ such that $$cn^{3/2}\leq \ex(n,F_t)\leq Cn^{3/2}.$$
\end{theorem}

In fact, our proof works for a slightly more general class of graphs. For graphs $G$ and $H$, the \emph{Cartesian product} $G\Box H$ is the graph whose vertex set is $V(G)\times V(H)$ and in which $(u,v)$ is adjacent to $(u',v')$ if and only if $u=u'$ and $vv'\in E(H)$ or $uu'\in E(G)$ and $v=v'$. Writing $P_t$ for the path with $t$ vertices, our Theorem \ref{thm:turan for grid} states that $\ex(n,P_t\Box P_t)=\Theta_t(n^{3/2})$. We can prove the following more general version.

\begin{theorem} \label{thm:turan for products}
    For any tree $T$ and any path $P$ (each with at least one edge), there exist positive real numbers $c$ and $C$ such that $$cn^{3/2}\leq \ex(n,T\Box P)\leq Cn^{3/2}.$$
\end{theorem}

\noindent
Since $T\Box P$ contains the $4$-cycle as a subgraph, the lower bound in Theorem \ref{thm:turan for products} follows from $\ex(n,C_4)=\Theta(n^{3/2})$ (see, e.g., \cite{KST54}).

The proof of Theorem \ref{thm:turan for grid} consists of two parts. First, we prove that $\ex(n, F_t) = O_t(n^{3/2} (\log n)^t)$. Then we get rid of the $(\log n)^t$ factor using the tensor power trick. This technique was used by Alon and Ruzsa \cite{AlRu99} to give an elementary proof of the celebrated Sidorenko’s conjecture for trees, which implies the Blakley-Roy matrix inequality \cite{BlPr65}. The technique has also been used in many other areas, see, e.g.,  Tao \cite{Tao08}, who has collected a number of these applications. While the tensor power trick is usually used to remove a constant factor, in this problem we manage to remove a factor which is polylogarithmic in $n$. 
Our way of using this tool somewhat differs from previous applications. Notably, we cannot deduce the correct bound on $\ex(n, F_t)$ by proving the aforementioned weaker bound and applying it as a black box to the $k^{th}$ tensor power $G^k$ of the given graph $G$; instead we work directly with $G^k$ while keeping in mind the original graph $G$. To the best of our knowledge, this is the first application of the tensor power trick to Tur\'an-type problems.

\subsection{An outline of the proof}

In this section we give a sketch of the proof of $\ex(n,F_t)=O_t(n^{3/2}(\log n)^{t})$. Let $G$ be an $n$-vertex graph with $\alpha n^{3/2}$ edges, where $\alpha=(\log n)^t$. We want to show that $G$ contains the $t\times t$ grid as a subgraph. We may assume by standard reduction results that $G$ is almost regular, that is, $\Delta(G)\leq K\delta(G)$ for some absolute constant $K$.

An important element of the proof will be a careful count of the $2\times t$ grids (i.e. subgraphs isomorphic to $P_2\Box P_t$) in $G$. This graph will be the building block for our $t\times t$ grids. More precisely, note that we can obtain an $(s+1)\times t$ grid by gluing a $2\times t$ grid to an $s\times t$ grid along some side of length $t-1$. Thus, by successively attaching $2\times t$ grids, we end up with a $t\times t$ grid as long as the attached $2\times t$ grids do not intersect the already existing grid in any vertex other than the ones on the boundary that are to be glued.

It is not too hard to see that $G$ contains roughly at least as many $2\times t$ grids as a random graph with the same edge density, namely $\alpha^{3t-2}n^{t/2+1}$. By the maximum degree condition, the number of paths of length $t-1$ in $G$ is at most about $\alpha^{t-1}n^{t/2+1/2}$ (ignoring absolute constant factors). This is promising as it shows that on average a path of length $t-1$ can be extended to many (at least $\alpha^{2t-1}n^{1/2}$) $2\times t$ grids. As long as we can make sure that not too many of these extensions contain any fixed vertex, we can conclude that $G$ contains a $t\times t$ grid. However, this is a significant obstacle since in general we do not have any good upper bound on the number of extensions of a path to a $2\times t$ grid which contain a given vertex. For this reason, we will look at a certain subfamily $\mathcal{F}$ of all $2\times t$ grids such that $|\mathcal{F}|$ is still large enough, but for which we will be able to bound efficiently the number of members of $\mathcal{F}$ that extend a fixed path of length $t-1$ and contain a fixed vertex.

In order to describe $\mathcal{F}$, it is helpful to label the vertices of the $2\times t$ grid. Let us call them $x_1,x_2,\dots,x_t,y_1,y_2,\dots,y_t$ where $x_1x_2\dots x_t$ and $y_1y_2\dots y_t$ are paths of length $t-1$, and $x_iy_i$ is an edge for every $i$. Now $\mathcal{F}$ will consist of those $2\times t$ grids in $G$ for which the codegree of $x_{i+1}$ and $y_i$ (in $G$) is at most $s_i$ for each $1\leq i\leq t-1$, where $s_1,s_2,\dots,s_{t-1}$ depend on the graph $G$. We claim that for suitable choices of $s_1,\dots,s_{t-1}\geq \alpha$, $|\mathcal{F}|$ is at least about $n(\alpha n^{1/2})^ts_1s_2\dots s_{t-1}/(\log n)^{t-1}$. Note that there is a trade-off between the control on the co-degrees and the lower bound on $|\mathcal{F}|$: the weaker control we have on the co-degrees, the bigger $\mathcal{F}$ is guaranteed to be.

Before we explain how we obtain the lower bound on $|\mathcal{F}|$, let us see how this guarantees that a $t\times t$ grid can be found.
The number of paths of length $t-1$ in $G$ is about $n(\alpha n^{1/2})^{t-1}$, so on average a path of length $t-1$ extends to at least $\alpha n^{1/2} s_1s_2\dots s_{t-1}/(\log n)^{t-1}$ members of $\mathcal{F}$. What we need to show is that only a small proportion of those can contain a fixed vertex. For this, fix the (images of the) vertices $x_1,\dots,x_t$ and let us count how many members of $\mathcal{F}$ extend these vertices which map $y_j$ to a fixed vertex $u\in V(G)$ where $2\leq j\leq t$ is given (the case $j=1$ is slightly different but arguably easier). Note that $y_1$ has to be a neighbour of $x_1$, so there are at most $\Delta(G)\approx \alpha n^{1/2}$ choices for it. Then $y_2$ has to be a common neighbour of $x_2$ and $y_1$, both of which are already fixed and need to have codegree at most $s_1$ by the definition of $\mathcal{F}$, so there are at most $s_1$ choices for $y_2$. Similarly, since $y_i$ needs to be a common neighbour of $x_i$ and $y_{i-1}$, once $y_1,y_2,\dots,y_{i-1}$ have been chosen, there are at most $s_{i-1}$ choices for $y_i$. Moreover, there is a unique choice for $y_j$ as we are counting only those extensions which map $y_j$ to $u$. Altogether, we get at most $\alpha n^{1/2}s_1s_2\dots s_{j-2}s_j\dots s_{t-1}$ extensions mapping $y_j$ to $u$. Since $s_{j-1}\geq \alpha=(\log n)^t$, we have that $\alpha n^{1/2}s_1s_2\dots s_{j-2}s_j\dots s_{t-1}=o(\alpha n^{1/2} s_1s_2\dots s_{t-1}/(\log n)^{t-1})$, so the number of extensions of any $x_1,\dots,x_t$ to a member of $\mathcal{F}$ mapping $y_j$ to $u$ is indeed negligible compared to the typical number of extensions of a fixed $x_1,\dots,x_t$ to a member of $\mathcal{F}$. Moreover, since the conditions on $\mathcal{F}$ are symmetrical with respect to $x$ and $y$, only a small proportion of all extensions of a given embedding of the $y_i$'s to a member of $\mathcal{F}$ use any given vertex. It is not hard to see that this implies that a $t\times t$ grid can be found in $G$.

We conclude this outline by sketching how to find enough $2\times t$ grids satisfying the codegree conditions. We first choose $x_1$ and $y_1$ for which there are $n\cdot \alpha n^{1/2}$ possibilities. Then we choose $x_2$ to be a neighbour of $x_1$ which can be done in $\alpha n^{1/2}$ many ways. By dyadic pigeonholing, there is some $s_1$ such that for at least $1/\log n$ proportion of all the partial $2\times t$ grids that were constructed so far, the codegree of $x_2$ and $y_1$ is around $s_1$ (up to a constant factor). For each of these partial grids, we can choose $y_2$ in around $s_1$ many ways. Then we choose $x_3$ to be an arbitrary neighbour of $x_2$ and again apply dyadic pigeonholing to keep at least $1/\log n$ proportion of our partial grids in all of which $x_3$ and $y_2$ have codegree around $s_2$. Continuing this process, we end up with the desired number of $2\times t$ grids. In turning this into a rigorous proof, the difficulty is making sure that $s_1,s_2,\dots,s_{t-1}$ are at least $\alpha$ (or even that there is always at least one choice for $y_i$ which is different from $x_1,y_1,x_2,y_2,\dots,x_{i-1},y_{i-1}$).

\textbf{Notation.} For a graph $G$, we use $\delta(G), \bar{d}(G)$ and $\Delta(G)$ to denote the minimum, average and maximum degree of $G$, respectively. For vertices $u,v\in V(G)$, we write $d_G(u,v)$ for the codegree of $u$ and $v$ in $G$. We omit floor and ceiling signs whenever this does not affect the argument. All logarithms are to the base $2$.
\section{The proof of Theorem \ref{thm:turan for products}}

A graph $G$ is called $K$-almost-regular if $\Delta(G) \leq K\delta(G)$. A well-known regularization lemma of Erd\H{o}s and Simonovits \cite{ESi69} allows us to restrict our attention to almost-regular graphs. We use the following version of the lemma proved by Jiang and Seiver.

\begin{lemma}[Jiang--Seiver \cite{JS12}] \label{lem:almost-regularization}
  Let $\eps,c$ be positive reals, where $\eps<1$ and $c\geq 1$. Let $n$ be a positive integer that is sufficiently large as a function of $\eps$. Let $G$ be a graph on $n$ vertices with $e(G)\geq cn^{1+\eps}$. Then $G$ contains a $K$-almost-regular subgraph $G'$ on $m\geq n^{\frac{\eps-\eps^2}{2+2\eps}}$ vertices such that $e(G')\geq \frac{2c}{5}m^{1+\eps}$ and $K=20\cdot 2^{\frac{1}{\eps^2}+1}$.
\end{lemma}

As discussed in the proof outline, the first step of our argument is to find many $2 \times t$ grids with bounds on the codegree of certain pairs. In order to do this, we would like to assume that for each edge $uv,$ there are $\Omega(\bar{d}(G))$ neighbours $w$ of $u$ such that $d(v, w) \ge C,$ where $C$ is some large constant. If $C$ is much larger than $t,$ then this allows us to extend a $2 \times s$ grid to a $2 \times (s+1)$ grid in $\Omega(\bar{d}(G) \cdot C)$ many ways. The following lemma allows us to find a large subgraph with the  property mentioned above.

\begin{lemma} \label{lem:cleaned subgraph}
    Let $G$ be an $n$-vertex graph with $\alpha n^{3/2}$ edges. Then $G$ has a subgraph $H$ on the same vertex set such that $e(H)\geq \frac{1}{2}\alpha n^{3/2}$ and for any $uv\in E(H)$, $u$ has at least $\frac{1}{8}\alpha n^{1/2}$ neighbours $w$ in $H$ such that $d_{H}(v,w)\geq \alpha^2/32$.
\end{lemma}

\begin{proof}
Let us define a sequence $G_0\supset G_1\supset G_2\supset \dots $ of graphs as follows. We set $G_0=G$. Having defined $G_i$, if there is a vertex $u\in V(G_i)$ with $1\leq d_{G_i}(u)\leq \frac{1}{4}\alpha n^{1/2}$, then choose such a vertex and let $G_{i+1}$ be the graph obtained from $G_i$ by deleting all edges incident to $u$. Call this deletion type 1. If no such vertex exists but there is an edge $uv\in E(G_i)$ for which $u$ has less than $\frac{1}{8}\alpha n^{1/2}$ neighbours $w$ in $G_i$ satisfying $d_{G_i}(v,w)\geq \alpha^2/32$, then let $G_{i+1}=G_i-uv$. Call this deletion type 2. If no such vertex or edge exists, then set $H=G_i$ and terminate the process.

Now for any $uv\in E(H)$, it follows immediately from the definition that $u$ has at least $\frac{1}{8}\alpha n^{1/2}$ neighbours $w$ in $H$ such that $d_{H}(v,w)\geq \alpha^2/32$. It remains to check that $e(H)\geq \frac{1}{2}\alpha n^{3/2}$, or equivalently that $e(G)-e(H)\leq \frac{1}{2}\alpha n^{3/2}$.

We shall prove that at most $\frac{1}{4}\alpha n^{3/2}$ edges are deleted in type 1 deletions and at most $\frac{1}{4}\alpha n^{3/2}$ edges are deleted in type 2 deletions. Indeed, there are at most $n$ type 1 deletion steps and each of them removes at most $\frac{1}{4}\alpha n^{1/2}$ edges, so it is clear that at most $\frac{1}{4}\alpha n^{3/2}$ edges are removed during type 1 deletions.

Since in each type 2 deletion, we remove precisely one edge, it suffices to prove that there are at most $\frac{1}{4}\alpha n^{3/2}$ such deletions throughout the process. Assume that the edge $uv$ gets deleted from $G_i$ because $u$ has less than $\frac{1}{8}\alpha n^{1/2}$ neighbours $w$ in $G_i$ such that $d_{G_i}(v,w)\geq \alpha^2/32$. Since no deletion of type 1 was applied to $G_i$, but there is an edge in $G_i$ incident to $u$, we have $d_{G_i}(u)>\frac{1}{4}\alpha n^{1/2}$. It follows that $u$ has more than $\frac{1}{8}\alpha n^{1/2}$ neighbours $w$ in $G_i$ such that $d_{G_i}(v,w)< \alpha^2/32$. For each such $v$, $u$ is a common neighbour of $v$ and $w$, so $d_{G_{i+1}}(v,w)=d_{G_i}(v,w)-1$. However, the condition $d_{G_i}(v,w)<\alpha^2/32$ shows that any pair $(v,w)$ of vertices can ``play this role" in at most $\alpha^2/32$ type~2 deletions. Thus, there are at most $\frac{n^2\cdot \alpha^2/32}{\frac{1}{8}\alpha n^{1/2}}=\frac{1}{4}\alpha n^{3/2}$ type 2 deletions in the process.
\end{proof}

Combining the previous two lemmas, we have the following.

\begin{lemma} \label{lem:final-cleaning}
    Let $G$ be an $n$-vertex graph with $\alpha n^{3/2}$ edges, where $\alpha \ge 10$ and $n$ is sufficiently large. Then $G$ has a subgraph $H$ on $m \ge n^{1/12}$ vertices such that for some $\alpha' \ge \alpha / 20,$ $e(H) \ge \alpha' m^{3/2}, \Delta(H) \le 12000 \alpha' m^{1/2}$ and for any $uv \in E(H),$ $u$ has at least $\alpha' m^{1/2}$ neighbours $w$ in $H$ satisfying $d_H(v, w) \ge \alpha'.$
\end{lemma}
\begin{proof}
    We apply Lemma~\ref{lem:almost-regularization} to $G$ to obtain a $K$-almost-regular subgraph $H_0$ on $m \ge n^{1/12}$ vertices with $e(H_0) = \alpha_0 m^{3/2},$ where $K = 640$ and $\alpha_0 \ge \frac{2}{5} \alpha \ge 4$. Set $\alpha' = \alpha_0 / 8 \ge \alpha / 20$. Applying Lemma~\ref{lem:cleaned subgraph} to $H_0$ gives us a subgraph $H \subseteq H_0$ such that $e(H) \ge \frac{1}{2} \alpha_0 m^{3/2} \ge \alpha' m^{3/2}$ and for any $uv \in E(H)$, $u$ has at least $\frac{1}{8} \alpha_0 m^{1/2} = \alpha' m^{1/2}$ neighbours $w$ in $H$ such that $d_H(v, w) \ge \alpha_0^2 / 32 \ge \alpha'.$ Finally, $\Delta(H) \le \Delta(H_0) \le K \delta(H_0) \le K \bar{d}(H_0) = 2K\alpha_0 m^{1/2} = 16K \alpha' m^{1/2} < 12000 \alpha' m^{1/2}$, as claimed.
\end{proof}

By Lemma~\ref{lem:final-cleaning}, in order to prove Theorem~\ref{thm:turan for products} it suffices to consider host graphs $G$ on $n$ vertices, where $n$ is sufficiently large, which satisfy the following, for a sufficiently large constant $\alpha$:
\begin{enumerate}[label=\alph*), itemsep=0pt]
    \item \label{number-of-edges} $e(G) \ge \alpha n^{3/2},$
    \item \label{max-degree} $\Delta(G) \le K \alpha n^{1/2}$, where $K=12000,$
    \item \label{neighbours-with-codegree} for any $uv \in E(G),$ $u$ has at least $\alpha n^{1/2}$ neighbours $w$ in $G$ such that $d_G(v, w) \ge \alpha$.
\end{enumerate}

From now on, we fix a large constant $\alpha$ to be chosen later and let $G$ be a graph on $n$ vertices, which satisfies \ref{number-of-edges}--\ref{neighbours-with-codegree}.

Let $G^k$ be the graph whose vertex set is $V(G)^k$ and in which $(u_1,\dots,u_k)$ and $(v_1,\dots,v_k)$ are adjacent if $u_jv_j\in E(G)$ for all $1\leq j\leq k$. This graph is called the $k^{th}$ tensor power of $G$. We use bold characters to denote vertices of the graph $G^k$ and for $\vx \in G^k, j \in [k],$ we denote by $\vx(j)$ the $j^{th}$ coordinate of $\vx$, which is a vertex of $G$.

\begin{definition}
    In a graph $H$, a \emph{$t$-ladder} is a $2t$-tuple $(x_1,y_1,\dots,x_t,y_t)\in V(H)^{2t}$ such that for each $1\leq i\leq t-1$, $x_ix_{i+1},y_iy_{i+1},x_iy_i\in E(H)$ and $x_ty_t\in E(H)$.
    
    For real numbers $s_1,\dots,s_{t-1}$, we call a $t$-ladder $(\vx_1,\vy_1,\dots,\vx_t,\vy_t)$ in $G^k$ \emph{$(s_1,s_2,\dots,s_{t-1})$-good} if
    \begin{itemize}
        \item for each $1\leq i\leq t-1$, $d_{G^k}(\vx_{i+1},\vy_i)\leq s_i$,
        \item for each $1\leq i\leq t-1,1\leq j\leq k$, $d_{G}(\vx_{i+1}(j),\vy_i(j))\geq \alpha$ and
        \item for each $1\leq j\leq k$, the vertices $\vx_1(j),\dots,\vx_t(j),\vy_1(j),\dots,\vy_t(j)$ are distinct.
    \end{itemize} 
\end{definition}

\begin{lemma} \label{lem:many ladders}
    There exist real numbers $s_1, s_2, \dots, s_{t-1} \ge 1$ for which the number of $(s_1, \dots, s_{t-1})$-good $t$-ladders in $G^k$ is at least
    $$\frac{\alpha^{tk} n^{(t/2+1)k} \prod_{i=1}^{t-1} s_i}{(4^{k+1} \log n^k)^{t-1}}.$$
\end{lemma}

\begin{proof}
We prove the statement by induction on $t$. Since $e(G)\geq \alpha n^{3/2}$, we have $e(G^k)\geq \alpha^k n^{3k/2}$, so the statement holds for $t=1$.
Assume we have already found real numbers $s_1,\dots,s_{t-2} \ge \alpha$ and at least
$$\frac{\alpha^{(t-1)k} n^{((t-1)/2 + 1)k} \prod_{i=1}^{t-2} s_i}{(4^{k+1} \log n^k)^{t-2}}$$
$(s_1,\dots,s_{t-2})$-good $(t-1)$-ladders in $G^k$. Note that for any such $(t-1)$-ladder $(\vx_1, \vy_1, \dots,\vx_{t-1},\vy_{t-1})$ and any $1\leq j\leq k$, we have $\vx_{t-1}(j)\vy_{t-1}(j)\in E(G)$. Hence, by \ref{neighbours-with-codegree}, $\vx_{t-1}(j)$ has at least $\alpha n^{1/2}$ neighbours $w$ in $G$ such that $d_{G}(w,\vy_{t-1}(j))\geq \alpha$. So $\vx_{t-1}$ has at least $(\frac{1}{2}\alpha n^{1/2})^k$ neighbours $\vx_t$ in $G^k$ such that for all $1\leq j\leq k$, we have that $d_G(\vx_t(j),\vy_{t-1}(j))\geq \alpha$ and $\vx_t(j)$ is distinct from $\vx_1(j),\dots,\vx_{t-1}(j),\vy_1(j),\dots,\vy_{t-1}(j)$.
So there are at least $$\frac{\alpha^{(t-1)k} n^{((t-1)/2 + 1)k} \prod_{i=1}^{t-2} s_i}{(4^{k+1} \log n^k)^{t-2}} \cdot \left(\frac{1}{2} \alpha n^{1/2}\right)^k$$ $(2t-1)$-tuples $(\vx_1,\vy_1,\vx_2,\dots,\vy_{t-1},\vx_t)$ in $G^k$ such that $(\vx_1,\vy_1,\vx_2,\dots,\vx_{t-1},\vy_{t-1})$ is an $(s_1,\dots,s_{t-2})$-good $(t-1)$-ladder, $\vx_t$ is a neighbour of $\vx_{t-1}$ in $G^k$ and for all $1\leq j\leq k$, we have that 
$\vx_t(j)$ is distinct from $\vx_1(j),\dots,\vx_{t-1}(j),\vy_1(j),\dots,\vy_{t-1}(j)$ and
$d_G(\vx_t(j),\vy_{t-1}(j))\geq \alpha$. Since $v(G^k) = n^k,$ by dyadic pigeonholing, there is some $s_{t-1}\geq 1$ (in fact, $s_{t-1} \ge \alpha^k$) such that in at least $1/\log n^k$ proportion of these tuples, we have $s_{t-1}/2<d_{G^k}(\vx_t,\vy_{t-1})\leq s_{t-1}$. For each such tuple, the number of ways to choose $\vy_t$ to be a common neighbour of $\vx_t$ and $\vy_{t-1}$ in $G^k$ is precisely $\prod_{j=1}^k d_G(\vx_t(j),\vy_{t-1}(j))$, and (since $d_G(\vx_t(j),\vy_{t-1}(j))\geq \alpha\geq 4t$), at least $\prod_{j=1}^k \frac{1}{2}d_G(\vx_t(j),\vy_{t-1}(j))$ of these choices satisfy that $\vy_t(j)$ is distinct from $\vx_1(j),\vy_1(j), \vx_2(j), \vy_2(j), \dots,\vx_t(j)$. We have $\prod_{j=1}^k \frac{1}{2}d_G(\vx_t(j),\vy_{t-1}(j))=2^{-k}d_{G^k}(\vx_t,\vy_{t-1})\geq 2^{-(k+1)}s_{t-1}$, so there are at least
$$\frac{\alpha^{(t-1)k} n^{((t-1)/2 + 1)k} \prod_{i=1}^{t-2} s_i}{(4^{k+1} \log n^k)^{t-2}} \cdot \left(\frac{1}{2} \alpha n^{1/2}\right)^k \cdot \frac{s_{t-1}}{2^{k+1}\log n^k} \ge \frac{\alpha^{tk} n^{(t/2+1)k} \prod_{i=1}^{t-1} s_i}{(4^{k+1} \log n^k)^{t-1}}$$
$(s_1,\dots,s_{t-1})$-good $t$-ladders in $G^k$. This completes the induction step.
\end{proof}
Recall that in the proof outline we showed that for a fixed path $(x_1, \dots, x_t)$ and a fixed additional vertex $u$, there are few ways to extend $(x_1, \dots, x_t)$ to (an analogue of) a good ladder containing $u$. Directly translating this argument to $G^k$ would give us that for any $(t-1)$-path $(\vx_1, \dots, \vx_t)$ in $G^k$ and any $\vu \in V(G^k)$, there are few ways to extend $(\vx_1, \dots, \vx_t)$ to an $(s_1, \dots, s_{t-1})$-good $t$-ladder containing $\vu$. However, and this is where the advantage of working with $G^k$ is revealed, we can draw a similar conclusion even if we count the extensions that contain some fixed vertices of $G$ in some subset $J \subseteq [k]$ (rather than all) of the coordinates. Naturally, the bound on the number of extensions gets stronger the larger the set $J$. 

\begin{lemma} \label{lem:few bad ladders}
    Let $s_1, \dots, s_{t-1} \ge 1, 1\leq \ell\leq t$, let $J\subset [k]$ and consider fixed vertices $u_j\in V(G)$ for every $j\in J$.
    
    \begin{enumerate}[label=(\alph*)]
        \item For any fixed $\vx_1,\dots,\vx_t\in V(G^k)$, the number of $(s_1,\dots,s_{t-1})$-good $t$-ladders in $G^k$ of the form $(\vx_1,\vy_1,\vx_2,\vy_2,\dots,\vx_t,\vy_t)$ with $\vy_\ell(j)=u_j$, for all $j\in J$, is at most $\Delta(G)^k \prod_{i=1}^{t-1} s_i \cdot \alpha^{-|J|}$. \label{statement a}
        
        \item For any fixed $\vy_1,\dots,\vy_t\in V(G^k)$, the number of $(s_1,\dots,s_{t-1})$-good $t$-ladders in $G^k$ of the form $(\vx_1,\vy_1,\vx_2,\vy_2,\dots,\vx_t,\vy_t)$ with $\vx_\ell(j)=u_j$, for all $j\in J$, is at most $\Delta(G)^k \prod_{i=1}^{t-1} s_i \cdot \alpha^{-|J|}$. \label{statement b}
    \end{enumerate}
\end{lemma}

\begin{proof}
We first prove \ref{statement a}. Let us first assume that $\ell=1$. Note that for every $j\in [k]$, $\vy_1(j)$ is a neighbour of $\vx_1(j)$ in $G$ and for each $j\in J$, $\vy_1(j)=u_j$. Hence, there are at most $\Delta(G)^{k-|J|}$ possibilities for $\vy_1$. Furthermore, for each $2\leq i\leq t$, $\vy_i$ is a common neighbour of $\vx_i$ and $\vy_{i-1}$ in $G^k.$ Since in every $(s_1, \dots, s_{t-1})$-good $t$-ladder, $d_{G^k}(\vx_i,\vy_{i-1})\leq s_{i-1}$, there are at most $s_{i-1}$ possibilities for $\vy_i$ given $\vy_1,\dots,\vy_{i-1}$. Altogether we get at most $\Delta(G)^{k-|J|} \prod_{i=1}^{t-1} s_i \leq \Delta(G)^k \prod_{i=1}^{t-1} s_i \cdot \alpha^{-|J|}$ possibilities, where we used $\Delta(G) \ge \alpha$. This completes the proof of \ref{statement a} in the case $\ell = 1$.

Assume now $\ell \ge 2.$ Since $\vy_1$ is a neighbour of $\vx_1$, there are at most $\Delta(G^k)=\Delta(G)^k$ choices for it. For any $2\leq i\leq \ell-1$, $\vy_i$ is a common neighbour of $\vx_i$ and $\vy_{i-1}$ in $G^k$, so given $\vy_1,\dots,\vy_{i-1}$, there are at most $s_{i-1}$ choices for it because we are only counting $(s_1, \dots, s_{t-1})$-good $t$-ladders in $G^k$. Given $\vy_1,\dots,\vy_{\ell-1}$, there are at most $\prod_{j\in [k]\setminus J} d_G(\vx_\ell(j),\vy_{\ell-1}(j))$ choices for $\vy_\ell$ since $\vy_\ell(j)$ is given for each $j\in J$.
Note that
$$\prod_{j\in [k]\setminus J} d_G(\vx_\ell(j),\vy_{\ell-1}(j))=\frac{\prod_{j\in [k]} d_G(\vx_\ell(j),\vy_{\ell-1}(j))}{\prod_{j\in J} d_G(\vx_\ell(j),\vy_{\ell-1}(j))}=\frac{d_{G^k}(\vx_\ell,\vy_{\ell-1})}{\prod_{j\in J} d_G(\vx_\ell(j),\vy_{\ell-1}(j))}\leq \frac{s_{\ell-1}}{\alpha^{|J|}},$$
where the last inequality holds because we are counting $(s_1, \dots, s_{t-1})$-good $t$-ladders in $G^k$.

Finally, for any $\ell+1\leq i\leq t$, there are at most $s_{i-1}$ possibilities for $\vy_i$, given $\vy_1,\dots,\vy_{i-1}$, so \ref{statement a} is proved.

To prove part \ref{statement b}, observe the following symmetry. If $(\vx_1, \vy_1, \vx_2, \vy_2, \dots, \vx_t, \vy_t)$ is an $(s_1, \dots, s_{t-1})$-good $t$-ladder in $G^k$, then $(\vy_t, \vx_t, \vy_{t-1}, \vx_{t-1}, \dots, \vy_1, \vx_1)$ is an $(s_{t-1}, s_{t-2}, \dots, s_1)$-good $t$-ladder in $G^k.$ Now, \ref{statement b} follows from \ref{statement a}.
\end{proof}

We are now ready to prove Theorem~\ref{thm:turan for products}. We build an auxiliary graph whose vertices are the copies of $P_t$ in $G^k$ and the edges correspond to $(s_1, \dots, s_{t-1})$-good $t$-ladders, for an appropriate choice of $s_1, \dots, s_{t-1}.$ An embedding of $T$ into this graph gives $k$ homomorphic copies of $T \Box P_t$ in $G$, one for each coordinate of $G^k$. Using the previous lemma allows us to greedily find an embedding such that at least one of these copies has distinct vertices, that is, the vertices of this coordinate form a genuine copy of $T \Box P_t$ in $G$. We make these arguments precise below.

\begin{proof}[Proof of Theorem~\ref{thm:turan for products}]
    Let $T$ be a tree on $r \ge 2$ vertices and let $t \ge 2.$ Recall that it is enough to prove that $T \Box P_t$ is a subgraph of a given graph $G$ satisfying \ref{number-of-edges}--\ref{neighbours-with-codegree}. Let $\alpha = (16K)^{r^2t^3}$ and let $k = k(t, r, n)$ be large enough so that $2^k > k \cdot 4t^2r \log n$. Let $s_1, \dots, s_{t-1}$ be real numbers given by Lemma~\ref{lem:many ladders} and let $\cF$ be the family of $(s_1, \dots, s_{t-1})$-good $t$-ladders in $G^k$. We define a graph $\cG$ as follows. The vertex-set of $\cG$ is the set of $t$-tuples $(\vx_1, \dots, \vx_t) \in V(G^k)^t$ forming a path in $G^k$ such that for all $j \in [k],$ the vertices $\vx_1(j), \dots, \vx_t(j)$ are distinct. We let $(\vx_1, \dots, \vx_t)$ and $(\vy_1, \dots, \vy_t)$ form an edge in $\cG$ if $(\vx_1, \vy_1, \vx_2, \vy_2 \dots, \vx_t, \vy_t)$ or $(\vy_1, \vx_1, \vy_2, \vx_2 \dots, \vy_t, \vx_t)$ is in $\cF$.
    
    The vertices of $\cG$ correspond to paths of length $t-1$ in $G^k$, so $|V(\cG)| \le |V(G^k)| \Delta(G^k)^{t-1} = |V(G)|^k \Delta(G)^{k(t-1)} \le n^k (K \alpha n^{1/2})^{k(t-1)}$. Also, any element of $\cF$ gives an edge in $\cG$ and at most two elements of $\cF$ give the same edge, so
    $$|E(\cG)| \geq |\cF|/2 \ge \frac{\alpha^{tk} n^{(t/2+1)k} \prod_{i=1}^{t-1} s_i}{2(4^{k+1} \log n^k)^{t-1}}.$$
    By successively removing vertices of degree less than $\bar{d}(\cG)/2$ from $\cG$, we obtain a non-empty subgraph $\cH$ of $\cG$ with
    \begin{equation}
        \delta(\cH) \ge \frac{1}{2} \bar{d}(\cG) = \frac{|E(\cG)|}{|V(\cG)|} \ge \frac{(\alpha n^{1/2})^k \prod_{i=1}^{t-1} s_i}{2(4^{k+1}K^k \log n^k)^{t-1}}. \label{eqn:min degree lower bound}
    \end{equation}
    
Let $w_1, w_2, \dots, w_r$ be a $1$-degenerate ordering of the vertices of $T$, that is, an ordering such that for all $2 \le p \le r,$ there is a unique $z < p$ with $w_pw_z \in E(T)$.

\medskip

\noindent \emph{Claim.} There exists an embedding of $T$ into $\cH$ which satisfies the following. Writing $(\vx_1^p, \vx_2^p, \dots, \vx_t^p)$ for the vertex of $\cH$ into which $w_p$ is embedded, for every $1 \le q < p \le r$ and $1 \le i, \ell \le t,$ there are at most $\frac{k}{(rt)^2}$ values of $j \in [k]$ such that $\vx_\ell^p(j) = \vx_i^q(j)$.

\medskip

Assume we are given the embedding as in the claim and let us show how it implies the theorem. For every $1 \le p \le r$ and every $j \in [k],$ since $(\vx_1^p, \vx_2^p, \dots, \vx_t^p) \in V(\cH),$ it follows that the vertices $\vx_1^p(j), \vx_2^p(j), \dots, \vx_t^p(j)$ form a path in $G.$ Similarly, whenever $w_pw_z \in E(T),$ the vertices $(\vx_1^p, \vx_2^p, \dots, \vx_t^p)$ and $(\vx_1^q, \vx_2^q, \dots, \vx_t^q)$ form an edge in $\cH$, implying that $(\vx_i^p(j), \vx_i^z(j)) \in E(G)$ for all $i \in[t], j \in [k]$. Hence, if for some $j \in [k],$ the vertices $\vx_i^p(j), 1 \le p \le r, 1 \le i \le t,$ are all distinct, then they form a copy of $T \Box P_t$ in $G$. The existence of such an index $j$ follows by simple counting. Indeed, summing over all $1 \le p < q \le r, 1 \le i, \ell \le t,$ there are at most ${r \choose 2} t^2 \frac{k}{(rt)^2} < k$ indices $j$ for which the vertices $\vx_i^p(j), 1 \le p \le r, 1 \le i \le t$ are not all distinct.

\medskip

\noindent \emph{Proof of Claim.} We iteratively find the desired embedding. We map $w_1$ to an arbitrary vertex $(\vx^1_1, \dots, \vx^1_t) \in \cH$. Now, fix $2 \le p \le r$ and suppose that for each $1\leq q\leq p-1$, $(\vx_1^q,\dots,\vx_t^q)$ has already been found. Let $1 \le z < p$ be the unique index such that $w_pw_z \in E(T)$. We choose $(\vx_1^p,\dots,\vx_t^p)$ to be a suitable neighbour of $(\vx_1^z\dots,\vx_t^z)$ in $\cH$. For any $1\leq q\leq p-1$, $1\leq i,\ell \leq t$ and $J\subset [k]$ of size $\lceil k/(rt)^2\rceil$, Lemma \ref{lem:few bad ladders} shows that $(\vx_1^z,\dots,\vx_t^z)$ has at most $\frac{2\Delta(G)^k \prod_{i=1}^{t-1} s_i}{\alpha^{\lceil k / (rt)^2 \rceil}}$ neighbours $(\vx_1^p,\dots,\vx_t^p)$ in $\cH$ with $\vx_\ell^p(j)=\vx_i^q(j)$ for all $j\in J$. Summing over all choices for $q$, $i$, $\ell$ and $J$, we find that the number of neighbours of $(\vx_1^z,\dots,\vx_t^z)$ that are not suitable is at most
\begin{align*}
    t^2r 2^k \cdot \frac{2\Delta(G)^k \prod_{i=1}^{t-1} s_i}{\alpha^{\lceil k / (rt)^2 \rceil}}
    &\le 2t^2r \frac{(2K\alpha n^{1/2})^k \prod_{i=1}^{t-1} s_i}{\alpha^{\lceil k / (rt)^2 \rceil}}
    \le 2t^2r \frac{(\alpha n^{1/2})^k \prod_{i=1}^{t-1} s_i}{(16K)^{k(t-1)}}\\
    &\le \frac{4t^2 r (\log n^k)^{t-1}}{2^{k(t-1)}} \cdot \frac{(\alpha n^{1/2})^k \prod_{i=1}^{t-1} s_i}{2(4^{k+1}K^k \log n^k)^{t-1}} < \delta(H),
\end{align*}
where in the second inequality we used that $\alpha^{\lceil k / (rt)^2 \rceil} \ge (16K)^{kt}$ and in the last inequality we used $2^k > 4t^2r k \log n$ as well as (\ref{eqn:min degree lower bound}). It follows that $(\vx_1^z, \dots, \vx_t^z)$ has a suitable neighbour in $\cH$, proving the claim.
\end{proof}

\section{Concluding remarks}

We conjecture the following generalization of Theorem \ref{thm:turan for products}.

\begin{conjecture} \label{conj:product}
    For any two trees $T$ and $S$ (each with at least one edge), there exist positive real numbers $c$ and $C$ such that
    $$cn^{3/2}\leq \ex(n,T\Box S)\leq Cn^{3/2}.$$
\end{conjecture}

\noindent
Note that if $T$ and $S$ are trees, then $T\Box S$ is $2$-degenerate. This observation provides some evidence towards Conjecture \ref{conj:product} since the conjecture of Erd\H os \cite{Erd67}, mentioned in the introduction,  asserts that if $H$ is an $r$-degenerate bipartite graph, then $\ex(n,H)=O(n^{2-1/r})$.

It would also be very interesting to study grids of higher dimension. For positive integers $d$ and $t$, the $d$-dimensional grid $F_t^{(d)}$ has vertex set $[t]^d$ with two vertices joined by an edge if they differ in exactly one coordinate and in that coordinate they differ by exactly one. Since $F_t^{(d)}$ is $d$-degenerate, the following conjecture is natural.

\begin{conjecture}
    For any positive integers $d$ and $t$, there is a constant $C$ such that
    $$\ex(n,F_t^{(d)})\leq Cn^{2-1/d}.$$
\end{conjecture}

\noindent
It is likely that this bound is not tight and can be improved. On the other hand, there is the probabilistic lower bound 
$\ex(n,F_t^{(d)})=\Omega\big(n^{2-\frac{t^d-2}{(t^d-t^{d-1})d-1}}\big)$, which is close to $n^{2-1/d}$ when $t$ is large.

It would also be interesting to determine the correct dependence of $\ex(n,F_t)$ on $t$. Our proof implies that the Tur\'an number of $F_t$ is at most $e^{O(t^5)} n^{3/2}$. On the other hand, the following construction shows that $\ex(n,F_t)\geq c t^{1/2} n^{3/2}$. Let $G_0$ be a $C_4$-free graph with $\frac{n}{t-1}$ vertices and about $(\frac{n}{t-1})^{3/2}$ edges. Let $G$ be the $(t-1)$-blowup of $G_0$. Then $G$ is an $n$-vertex graph and $e(G)\approx (t-1)^2 (\frac{n}{t-1})^{3/2}\geq ct^{1/2}n^{3/2}$. We claim that $G$ does not contain a $t\times t$ grid. Indeed, since $G_0$ is $C_4$-free, any $4$-cycle in $G$ must have an opposite pair of vertices which come from the same vertex of $G_0$. It is known (see, e.g. \cite{diagonals}) that if a diagonal is placed in each unit square of a $t\times t$ grid, then there is a path along these diagonals from one side of the grid to the opposite side (i.e. either from top to bottom or from left to right). Note that this path contains at least $t$ vertices of the grid. Hence, a $t\times t$ grid in $G$ would have to contain at least $t$ vertices which come from the same vertex in $G_0$, showing that such a grid cannot be found in $G$.

Finally, we remark that we can apply the tensor power trick similarly as it is done in this paper to another extremal problem (i.e., working directly with the $k$-th power rather than with the original graph/hypergraph). This problem asks to determine the asymptotics of the Tur\'an number of the $r$-uniform hypergraph  $K_{2,t}^{(r)}$ whose vertex set consists of disjoint sets $X,Y_1,\dots,Y_t$, where $|X| = 2$ and $|Y_1| = \dots = |Y_t| = r-1$, and whose edge set is $\{\{x\} \cup Y_i: x \in X, 1\leq i\leq t\}$. However, for that problem, there is a different proof which does not require tensorization, see \cite{BGJS} for more details. It would be interesting to find other instances where the tensor power technique can be applied in a similar way.


\end{document}